\documentclass{amsart}

\usepackage{amssymb}
\usepackage{amsmath}
\usepackage{amsthm}
\usepackage{amsfonts}

\usepackage{a4wide}
\usepackage{longtable}
\usepackage{multirow}

\newtheorem{lemma}{Lemma}[section]
\newtheorem{proposition}{Proposition}[section]
\newtheorem{theorem}{Theorem}[section]

\theoremstyle{definition}

\theoremstyle{remark}

\begin{document}

\title[The entropy of coclass one]{The entropy of coclass one \\ and of SmallGroup(729,45)}

\author{Helga Boyer von Berghof}
\address{Krenngasse 43\\8010 Graz\\Austria}
\email{helgaboyervonberghof@gmail.com}

\subjclass[2000]{11R37, 11R29, 11R11, 11R16, 11R20; 20D15, 20F14}

\keywords{Finite \(p\)-groups, (Schur\(+1\)) \(\sigma\)-groups, real probability measure, entropy, Galois-group, maximal unramified pro-\(p\)-extension, \(p\)-class field tower}

\date{Saturday, 5 September 2026}


\begin{abstract}
Investigating the descendant tree \(\mathcal{T}(\Delta)\) of \(3\)-groups \(G\) of maximal nilpotency-class,
with root \(\Delta=\mathrm{SmallGroup}(27,3)\),
we show for the first time that an entropy-value
\(H({\bf P}_r)=\sum\,{\bf P}_r(G)\cdot\lvert\log({\bf P}_r(G))\rvert\)
can also be determined for the real probability-distribution
\({\bf P}_r(G)\)
of all terminal vertices outside of the mainline on this infinite digraph.
These vertices arise as automorphism groups
\(\mathrm{Gal}(\mathrm{F}_3^\infty(K)/K)\)
of maximal unramified pro-\(3\)-extensions
\(\mathrm{F}_3^\infty(K)\)
of real-quadratic number fields
\(K=\mathbb{Q}(\sqrt{d})\), \(d>0\).
For the groups \(G\) of another descendant tree \(\mathcal{T}(N)\),
with root \(N=\mathrm{SmallGroup}(729,45)\)
and unbounded coclass,
which has been studied by L. Bartholdi and M. R. Bush
\cite{BaBu2007},
we compare the entropy \(H({\bf P}_r)\) of the real measure \({\bf P}_r(G)\)
with the entropy \(H({\bf P}_i)\) of the imaginary measure \({\bf P}_i(G)\).
\end{abstract}

\maketitle


\section{Entropy of descendant trees}
\label{s:Entropy}

\noindent
The notion of \textit{entropy} was introduced in number theory
by N. Minculete 
\cite{MiPo2011}
and D. Savin 
\cite{MiSa2023}.
Together with coauthors D. C. Mayer and V. Monescu
\cite{MMSM2026},
these authors were the first who determined the entropy 
\(H({\bf P})=\sum\,{\bf P}(G)\cdot\lvert\log({\bf P}(G))\rvert\)
of \textit{infinite probability-distributions}
\({\bf P}\).
Up to now, such investigations were confined to the probability-measure \({\bf P}(G)\)
of descendant trees with Schur \(\sigma\)-groups \(G\)
\cite{Ag1998}
as terminal vertices, which arise as Galois-groups
\(\mathrm{Gal}(\mathrm{F}_3^\infty(K)/K)\)
of maximal unramified pro-\(3\)-extensions
\(\mathrm{F}_3^\infty(K)\),
that is, \(3\)-class field towers,
of \textit{imaginary}-quadratic number fields
\(K=\mathbb{Q}(\sqrt{d})\)
with negative fundamental-discriminant \(d<0\)
\cite{KoVe1975}.
In the present article, we shall illuminate the analogous situation 
of (Schur\(+1\)) \(\sigma\)-groups \(G\)
\cite{BBH2021}
arising from \textit{real}-quadratic number fields \(K=\mathbb{Q}(\sqrt{d})\)
with positive fundamental-discriminant \(d>0\).
In both cases, our focus primarily lies on quadratic fields \(K\)
with \textit{elementary bicyclic} \(3\)-class group
\(\mathrm{Cl}_3(K)\simeq(\mathbb{Z}/3\mathbb{Z})^2\),
which is isomorphic to the commutator-quotient
\(G/G^\prime\)
of the tower group \(G\),
by the Artin reciprocity law of class field theory.


\subsection{Entropy of coclass one}
\label{ss:CoClassOne}
First we consider the descendant tree
\(\mathcal{T}(\Delta)\)
of metabelian \(3\)-groups of maximal nilpotency class
\(\mathrm{cl}(G)=e-1\)
in terms of the logarithmic order
\(\mathrm{lo}(G)=e\), i.e. \(\#G=3^e\),
which is called \textit{coclass one},
\(\mathrm{cc}(G)=\mathrm{lo}(G)-\mathrm{cl}(G)=1\).
Its root is the extra-special group \(\Delta=\langle 27,3\rangle\)
\cite{BEO2005}.
This tree with constant coclass is populated only by the groups
\(\mathrm{Gal}(\mathrm{F}_3^\infty(K)/K)\)
of \(3\)-class field towers
of real-quadratic fields \(K\).
For its analysis, we need the \textit{real probability-measure}
by Boston, Bush, Hajir (2021)
\cite{BBH2021}
in the Formula
\eqref{eqn:RealMeasure}
and the (Schur\(+1\)) \(\sigma\)-descendants \(G\) of \(\Delta\) in the Table
\ref{tbl:FormationLawsOne},
grouped by their \textit{transfer kernel type} (TKT).
The relation rank is denoted by \(r\).
\begin{equation}
\label{eqn:RealMeasure}
{\bf P}_r(G)=\frac{y(G)^{g+1}}{\#\mathrm{Aut}(G)\cdot\#G}\cdot(p^g)^{g+1}\cdot\prod_{k=1}^{g}\,\left(1-\frac{1}{p^k}\right)^2
\cdot\left(1-\frac{1}{p^{g+1}}\right)\cdot\prod_{k=1}^{g+1-r}\,\left(1-\frac{1}{p^k}\right)^{-1}
\end{equation}
In the relevant special case of
the smallest odd prime number \(p=3\)
and the generator rank \(g=2\),
the constant factor after the invariants which depend on \(G\) is given by:
\begin{equation}
\label{eqn:RealFactor}
\frac{2^9\cdot 13}{3^3} \text{ for } r=g+1 \text{ (Schur+1)}
\quad \text{ and } \quad 
\frac{2^8\cdot 13}{3^2} \text{ for } r=g \text{ (Schur)}.
\end{equation}


\begin{theorem}
\label{thm:CoClassOne}
The \textbf{entropy} of the descendant tree
\(\mathcal{T}(\Delta)\)
of all metabelian \(3\)-groups \(G\) of maximal nilpotency-class and fixed \textbf{coclass one},
\(\mathrm{cc}(G)=1\),
which are descendants of the extra-special \(3\)-group
\(\Delta=\langle 27,3\rangle\)
of exponent \(3\) without exceptions,
with respect to the normalized relative-measure
\({\bf P}_{rel}(G)=\frac{3^7}{2^7\cdot 13}\cdot{\bf P}_{r}(G)\)
associated to the real probability-measure in Formula
\eqref{eqn:RealMeasure},
is given by
\begin{equation}
\label{eqn:Entropy1}
H({\bf P}_{rel})
=-\sum\,{\bf P}_{rel}(G_n^i)\log({\bf P}_{rel}(G_n^i))
=\frac{1}{26}\biggl(41\log(3)-8\log(2)\biggr)
\approx{\bf 1.51915}.
\end{equation}
\end{theorem}
\begin{proof}
Using the Formula
\eqref{eqn:RealMeasure},
for \(r=g+1\),
and the information in Table
\ref{tbl:FormationLawsOne},
we calculate the measures
\({\bf P}_r(G)=\frac{y(G)^{2+1}}{\#\mathrm{Aut}(G)\cdot\#G}\cdot\frac{2^9\cdot 13}{3^3}\), grouped by TKTs: \\
a.1: \(\frac{(3^{n+1})^3}{2\cdot 3^{4n+4}\cdot 3^{2n+4}}\cdot\frac{2^9\cdot 13}{3^3}=\frac{2^8\cdot 13}{3^{3n+8}}\),
a.2: \(\frac{(3^{n+1})^3}{2\cdot 3^{4n+5}\cdot 3^{2n+4}}\cdot\frac{2^9\cdot 13}{3^3}=\frac{2^8\cdot 13}{3^{3n+9}}\),
a.3: \(\frac{(3^{n+1})^3}{2^2\cdot 3^{4n+4}\cdot 3^{2n+4}}\cdot\frac{2^9\cdot 13}{3^3}=\frac{2^7\cdot 13}{3^{3n+8}}\).

\renewcommand{\arraystretch}{1.2}
\begin{table}[ht]
\caption{Parametrized formation-laws of invariants, grouped by TKT}
\label{tbl:FormationLawsOne}
\begin{center}
\begin{tabular}{|c||c|c|c|}
\hline
 Transfer Kernel Type      & a.1, \(\varkappa\sim (0000)\)    &  a.2, \(\varkappa\sim (1000)\)   &  a.3, \(\varkappa\sim (2000)\)   \\
\hline
 State                     & \(n\ge 1\), \(3\) groups         & \(n\ge 0\), \(1\) group          & \(n\ge 0\), \(2\) groups         \\
 \(y(G_n^i)\)              & \(3^{n+1}\)                      & \(3^{n+1}\)                      & \(3^{n+1}\)                      \\
 \(\#\mathrm{Aut}(G_n^i)\) & \(2\cdot 3^{4n+4}\)              & \(2\cdot 3^{4n+5}\)              & \(2^2\cdot 3^{4n+4}\)            \\
 \(\#G_n^i\)               & \(3^{2n+4}\)                     & \(3^{2n+4}\)                     & \(3^{2n+4}\)                     \\
 \({\bf P}_r(G_n^i)\)      & \(\frac{2^8\cdot 13}{3^{3n+8}}\) & \(\frac{2^8\cdot 13}{3^{3n+9}}\) & \(\frac{2^7\cdot 13}{3^{3n+8}}\) \\
\hline
\end{tabular}
\end{center}
\end{table}

\noindent
Now we add all probability-measures
of metabelian \(3\)-groups of maximal nilpotency-class,
grouped by their TKTs,
and using the geometric series
\(\sum_{n=0}^{\infty}\,\frac{1}{27^n}
=\frac{1}{1-\frac{1}{27}}
=\frac{1}{\frac{26}{27}}
=\frac{27}{26}
=\frac{3^3}{2\cdot 13}\).
\begin{equation}
\label{eqn:a1}
3 \text{ times a.1}: \quad
\sum_{n=1}^{\infty}\,{\bf P}_r(G_n^1)=
\sum_{n=1}^{\infty}\,\frac{2^8\cdot 13}{3^{3n+8}}=
\frac{2^8\cdot 13}{3^{8}}\cdot\sum_{n=1}^{\infty}\,\left(\frac{1}{3^3}\right)^n=
\frac{2^8\cdot 13}{3^{8}}\cdot\frac{1}{2\cdot 13}=
\frac{2^7}{3^8},
\end{equation}
\begin{equation}
\label{eqn:a2}
1 \text{ times a.2}: \quad
\sum_{n=0}^{\infty}\,{\bf P}_r(G_n^2)=
\sum_{n=0}^{\infty}\,\frac{2^8\cdot 13}{3^{3n+9}}=
\frac{2^8\cdot 13}{3^{9}}\cdot\sum_{n=0}^{\infty}\,\left(\frac{1}{3^3}\right)^n=
\frac{2^8\cdot 13}{3^{9}}\cdot\frac{3^3}{2\cdot 13}=
\frac{2^7}{3^6},
\end{equation}
\begin{equation}
\label{eqn:a3}
2 \text{ times a.3}: \quad
\sum_{n=0}^{\infty}\,{\bf P}_r(G_n^3)=
\sum_{n=0}^{\infty}\,\frac{2^7\cdot 13}{3^{3n+8}}=
\frac{2^7\cdot 13}{3^{8}}\cdot\sum_{n=0}^{\infty}\,\left(\frac{1}{3^3}\right)^n=
\frac{2^7\cdot 13}{3^{8}}\cdot\frac{3^3}{2\cdot 13}=
\frac{2^6}{3^5}.
\end{equation}
Eventually, we add these three contributions,
taking into account their multiplicities, obtaining the 
\textbf{total measure of coclass one:}
\begin{equation}
\label{eqn:cc1}
\sum\,{\bf P}_{r}(G_n^i)=
3\cdot\frac{2^7}{3^8}+\frac{2^7}{3^6}+2\cdot\frac{2^6}{3^5}=
\frac{2^7}{3^7}\cdot(1+3+3^2)=
\frac{2^7\cdot 13}{3^7}=
\frac{\bf 1664}{\bf 2187}.
\end{equation}
With this result, we must normalize the probability-measure onto sum one, \({\bf P}_{rel}(G)=\frac{3^7}{2^7\cdot 13}\cdot{\bf P}_{r}(G)\):

\renewcommand{\arraystretch}{1.2}
\begin{table}[ht]
\caption{Normalized relative-measure and logarithms, grouped by TKT}
\label{tbl:RelativeMeasuresOne}
\begin{center}
\begin{tabular}{|c||c|c|c|}
\hline
 Transfer Kernel Type            & a.1, \(\varkappa\sim (0000)\)    &  a.2, \(\varkappa\sim (1000)\)   &  a.3, \(\varkappa\sim (2000)\)   \\
\hline
 State                           & \(n\ge 1\), \(3\) groups         & \(n\ge 0\), \(1\) group          & \(n\ge 0\), \(2\) groups         \\
 \({\bf P}_{rel}(G_n^i)\)        & \(\frac{2}{3^{3n+1}}\)           & \(\frac{2}{3^{3n+2}}\)           & \(\frac{1}{3^{3n+1}}\)           \\
 \(-\log({\bf P}_{rel}(G_n^i))\) & \((3n+1)\log(3)-\log(2)\)        & \((3n+2)\log(3)-\log(2)\)        & \((3n+1)\log(3)\)                \\
\hline
\end{tabular}
\end{center}
\end{table}

\noindent
We continue this proof of Theorem
\ref{thm:CoClassOne}
after an auxiliary Lemma
\ref{lem:FinitePart}
and Proposition
\ref{prp:GeometricVariant}.
\end{proof}

\noindent
Now we need the sum of certain infinite series,
variants of the geometric series,
which cannot be found in the standard literature on analysis.
\begin{lemma}
\label{lem:FinitePart}
For each real number \(q\in\mathbb{R}\), \(q\ne 1\),
and non-negative integers \(K\le N\), the following formula holds
for finite partial sums of a geometric series:
\begin{equation}
\label{eqn:FinitePart}
q^K+q^{K+1}+\cdots+q^{N-1}+q^N=\frac{q^K-q^{N+1}}{1-q}.
\end{equation}
\end{lemma}
\begin{proof}
Departing from the well-known polynomial identity
\[\left(q^N+q^{N-1}+\cdots+q^2+q+1\right)\cdot(q-1)=q^{N+1}-1,\]
we obtain for \(K\le N\):
\[\sum_{n=K}^N\,q^n
=q^K\cdot\sum_{n=0}^{N-K}\,q^n
=q^K\cdot\frac{1-q^{N+1-K}}{1-q}
=\frac{q^K-q^{N+1}}{1-q}. \qedhere\]
\end{proof}
\begin{proposition}
\label{prp:GeometricVariant}
For each bounded real number \(q\in\mathbb{R}\), \(\lvert q\rvert<1\),
the sum of the following variant of the infinite geometric series is given by:
\begin{equation}
\label{eqn:GeometricVariant}
\sum_{n=1}^\infty\,n\cdot q^n=\frac{q}{(1-q)^2}.
\end{equation}
\end{proposition}
\begin{proof}
We consider the partial sum, split it, and repeatedly use Lemma
\ref{lem:FinitePart}:
\begin{equation*}
\label{eqn:Splitting}
\begin{aligned}
\sum_{n=1}^N\,n\cdot q^n
&=1\cdot q^1+2\cdot q^2+\cdots+(N+1)\cdot q^{N-1}+N\cdot q^N \\
&=(q^1+\cdots+q^N)+(q^2+\cdots+q^N)+\cdots+(q^{N-1}+q^N)+(q^N) \\
&=\frac{q^1-q^{N+1}}{1-q}+\frac{q^2-q^{N+1}}{1-q}+\cdots+\frac{q^{N-1}-q^{N+1}}{1-q}+\frac{q^{N}-q^{N+1}}{1-q} \\
&=\frac{(q^1+q^2+\cdots+q^{N-1}+q^N)-N\cdot q^{N+1}}{1-q}=\frac{q-q^{N+1}}{(1-q)^2}-\frac{N\cdot q^{N+1}}{1-q}.
\end{aligned}
\end{equation*}
Finally we calculate the limit \(N\to\infty\), where the exponential function dominates every power:
\[\lim_{N\to\infty}\left(\sum_{n=1}^N\,n\cdot q^n\right)
=\frac{q}{(1-q)^2}-\lim_{N\to\infty}\left(\frac{q^{N+1}}{(1-q)^2}+\frac{N\cdot q^{N+1}}{1-q}\right)
=\frac{q}{(1-q)^2}. \qedhere\]
\end{proof}
\noindent
Applied to \(q=\frac{1}{27}\), we get
\(\sum_{n=1}^\infty\,\frac{n}{27^n}=\frac{\frac{1}{27}}{(1-\frac{1}{27})^2}=\frac{1}{3^3}\cdot\frac{3^6}{(2\cdot 13)^2}=\frac{3^3}{2^2\cdot 13^2}=\frac{27}{676}\),
needed in the sequel.
\begin{proof}
(of Theorem
\ref{thm:CoClassOne}
continued.)
We can now calculate the \textbf{entropy of the coclass one}:
{\small
\begin{equation*}
\label{eqn:EntropyOne}
\begin{aligned}
H({\bf P}_{rel})
&=-\sum\,{\bf P}_{rel}(G_n^i)\log({\bf P}_{rel}(G_n^i)) \\
&=3\sum_{n=1}^{\infty}\,\frac{2}{3^{3n+1}}\biggl((3n+1)\log3-\log2\biggr)
+\sum_{n=0}^{\infty}\,\frac{2}{3^{3n+2}}\biggl((3n+2)\log3-\log2\biggr) \\
&+2\sum_{n=0}^{\infty}\,\frac{1}{3^{3n+1}}(3n+1)\log3 \\
&=2\biggl(3\log3\cdot\sum_{n=1}^{\infty}\,\frac{n}{27^n}+(\log3-\log2)\cdot\sum_{n=1}^{\infty}\,\frac{1}{27^n}\biggr) \\
&+\frac{2}{9}\biggl(3\log3\cdot\sum_{n=0}^{\infty}\,\frac{n}{27^n}+(2\log3-\log2)\cdot\sum_{n=0}^{\infty}\,\frac{1}{27^n}\biggr) \\
&+\frac{2}{3}\biggl(3\log3\cdot\sum_{n=0}^{\infty}\,\frac{n}{27^n}+\log3\cdot\sum_{n=0}^{\infty}\,\frac{1}{27^n}\biggr) \\
&=\biggl(6+\frac{2}{3}+2\biggr)\log3\cdot\frac{27}{26^2}
+2(\log3-\log2)\cdot\frac{1}{26}
+\biggl(\frac{2}{9}(2\log3-\log2)+\frac{2}{3}\log3\biggr)\cdot\frac{27}{26} \\
&=\frac{1}{26}\biggl(41\log3-8\log2\biggr)
\approx{\bf 1.51915}. \qedhere 
\end{aligned}
\end{equation*}
}
\end{proof}

\newpage

\noindent
Figure
\ref{fig:MinDscTreeCc1}
shows the root region of the descendant tree
\(\mathcal{T}(\Delta)\),
embedded into the
coclass-\(1\) graph \(\mathcal{G}(3,1)\),
structured arithmetically with
minimal discriminants \(d_K\)
(underlined in bold font adjacent to surrounding ovals around vertices)
of real quadratic number fields \(K\),
up to logarithmic order \(\mathrm{lo}=10\).
Every other branch \(\mathcal{B}(e)\), with odd \(e=3,5,7,9,\ldots\),
consists of terminal (Schur\(+1\)) \(\sigma\)-groups,
except the branch root on the mainline.
The groups without an abelian maximal subgroup
in the double contour rectangle on the right hand side
can only be separated by means of \textit{deep transfers}
\cite{Ma2018}.
They start on branch \(\mathcal{B}(5)\)
without ground state \(n=0\) on branch \(\mathcal{B}(3)\).
The \textit{Artin pattern} \((\alpha,\varkappa_s)\) is the pair formed by 
the abelian quotient invariants (AQI) \(\alpha\)
and the (shallow) transfer kernel type (TKT) \(\varkappa=\varkappa_s\).
The notation \(G_a^e(z,w)\) with the parameters \(e;a,z,w\) is due to N. Blackburn and R. J. Miech.
The notation in angle brackets \(\langle\mathrm{ord},\mathrm{id}\rangle\) is due to
\cite{BEO2005},
\cite{GNO2006},
\cite{MAGMA2026},
\cite{MAGMA6561}.

\begin{figure}[ht]
\caption{Distribution of discriminants for \(\mathrm{Gal}(\mathrm{F}_3^\infty(K)/K)\) on the coclass-\(1\) tree \(\mathcal{T}(\Delta)\)}
\label{fig:MinDscTreeCc1}

{\tiny

\setlength{\unitlength}{0.8cm}
\begin{picture}(18,21)(-11,-20)

\put(-10,0.5){\makebox(0,0)[cb]{Order \(3^e\)}}

\put(-10,0){\line(0,-1){16}}
\multiput(-10.1,0)(0,-2){9}{\line(1,0){0.2}}

\put(-10.2,0){\makebox(0,0)[rc]{\(9\)}}
\put(-9.8,0){\makebox(0,0)[lc]{\(3^2\)}}
\put(-10.2,-2){\makebox(0,0)[rc]{\(27\)}}
\put(-9.8,-2){\makebox(0,0)[lc]{\(3^3\)}}
\put(-10.2,-4){\makebox(0,0)[rc]{\(81\)}}
\put(-9.8,-4){\makebox(0,0)[lc]{\(3^4\)}}
\put(-10.2,-6){\makebox(0,0)[rc]{\(243\)}}
\put(-9.8,-6){\makebox(0,0)[lc]{\(3^5\)}}
\put(-10.2,-8){\makebox(0,0)[rc]{\(729\)}}
\put(-9.8,-8){\makebox(0,0)[lc]{\(3^6\)}}
\put(-10.2,-10){\makebox(0,0)[rc]{\(2\,187\)}}
\put(-9.8,-10){\makebox(0,0)[lc]{\(3^7\)}}
\put(-10.2,-12){\makebox(0,0)[rc]{\(6\,561\)}}
\put(-9.8,-12){\makebox(0,0)[lc]{\(3^8\)}}
\put(-10.2,-14){\makebox(0,0)[rc]{\(19\,683\)}}
\put(-9.8,-14){\makebox(0,0)[lc]{\(3^9\)}}
\put(-10.2,-16){\makebox(0,0)[rc]{\(59\,049\)}}
\put(-9.8,-16){\makebox(0,0)[lc]{\(3^{10}\)}}

\put(-10,-16){\vector(0,-1){2}}

\put(-8,0.5){\makebox(0,0)[cb]{\(\alpha(1)=\)}}

\put(-8,0){\makebox(0,0)[cc]{\((1)\)}}
\put(-8,-2){\makebox(0,0)[cc]{\((1^2)\)}}
\put(-8,-4){\makebox(0,0)[cc]{\((21)\)}}
\put(-8,-6){\makebox(0,0)[cc]{\((2^2)\)}}
\put(-8,-8){\makebox(0,0)[cc]{\((32)\)}}
\put(-8,-10){\makebox(0,0)[cc]{\((3^2)\)}}
\put(-8,-12){\makebox(0,0)[cc]{\((43)\)}}
\put(-8,-14){\makebox(0,0)[cc]{\((4^2)\)}}
\put(-8,-16){\makebox(0,0)[cc]{\((54)\)}}

\put(-8,-17){\makebox(0,0)[cc]{\textbf{AQI}}}
\put(-8.5,-17.2){\framebox(1,18){}}

\put(3.7,-3){\vector(0,1){1}}
\put(3.9,-3){\makebox(0,0)[lc]{depth \(1\)}}
\put(3.7,-3){\vector(0,-1){1}}

\put(-5.1,-8){\vector(0,1){2}}
\put(-5.3,-7){\makebox(0,0)[rc]{period length \(2\)}}
\put(-5.1,-8){\vector(0,-1){2}}

\put(-0.1,-0.1){\framebox(0.2,0.2){}}
\put(-2.1,-0.1){\framebox(0.2,0.2){}}

\multiput(0,-2)(0,-2){8}{\circle*{0.2}}
\multiput(-2,-2)(0,-2){8}{\circle*{0.2}}
\multiput(-4,-4)(0,-2){7}{\circle*{0.2}}
\multiput(-6,-4)(0,-4){4}{\circle*{0.2}}

\multiput(2,-6)(0,-2){6}{\circle*{0.1}}
\multiput(4,-6)(0,-2){6}{\circle*{0.1}}
\multiput(6,-6)(0,-2){6}{\circle*{0.1}}

\multiput(0,0)(0,-2){8}{\line(0,-1){2}}
\multiput(0,0)(0,-2){8}{\line(-1,-1){2}}
\multiput(0,-2)(0,-2){7}{\line(-2,-1){4}}
\multiput(0,-2)(0,-4){4}{\line(-3,-1){6}}
\multiput(0,-4)(0,-2){6}{\line(1,-1){2}}
\multiput(0,-4)(0,-2){6}{\line(2,-1){4}}
\multiput(0,-4)(0,-2){6}{\line(3,-1){6}}

\put(0,-16){\vector(0,-1){2}}
\put(0.2,-17.6){\makebox(0,0)[lc]{infinite}}
\put(0.2,-18){\makebox(0,0)[lc]{mainline}}
\put(-1.0,-18.4){\makebox(0,0)[lc]{\(\mathcal{T}(\Delta)\subset\)}}
\put(0.2,-18.4){\makebox(0,0)[lc]{\(\mathcal{T}^1(\langle 9,2\rangle)\)}}

\put(0,0.1){\makebox(0,0)[lb]{\(=C_3\times C_3\)}}
\put(-2.5,0.1){\makebox(0,0)[rb]{\(C_9=\)}}
\put(-0.8,0.3){\makebox(0,0)[rb]{abelian}}
\put(0.1,-1.9){\makebox(0,0)[lb]{\(=G^3_0(0,0)=\Delta\)}}
\put(-2.5,-1.9){\makebox(0,0)[rb]{\(G^3_0(0,1)=\)}}
\put(-4.1,-3.9){\makebox(0,0)[lb]{\(=\mathrm{Syl}_3A_9\)}}
\put(-4,-4.5){\makebox(0,0)[cc]{\textbf{AQI} \quad \(\alpha(1)=(1^3)\)}}
\put(-5.5,-4.7){\framebox(2.9,0.5){}}

\put(0.2,-2.2){\makebox(0,0)[lt]{bifurcation from \(\mathcal{G}(3,1)\)}}
\put(1.5,-2.5){\makebox(0,0)[lt]{to \(\mathcal{G}(3,2)\)}}

\put(-2.1,0.1){\makebox(0,0)[rb]{\(\langle 1\rangle\)}}
\put(-0.1,0.1){\makebox(0,0)[rb]{\(\langle 2\rangle\)}}

\put(-0.2,-1){\makebox(0,0)[rc]{branch \(\mathcal{B}(2)\)}}
\put(-2.1,-1.9){\makebox(0,0)[rb]{\(\langle 4\rangle\)}}
\put(-0.1,-1.9){\makebox(0,0)[rb]{\(\langle 3\rangle\)}}

\put(-0.5,-3){\makebox(0,0)[cc]{\(\mathcal{B}(3)\)}}
\put(-6.1,-3.9){\makebox(0,0)[rb]{\(\langle 8\rangle\)}}
\put(-4.2,-3.9){\makebox(0,0)[rb]{\(\langle 7\rangle\)}}
\put(-2.1,-3.9){\makebox(0,0)[rb]{\(\langle 10\rangle\)}}
\put(-0.1,-3.9){\makebox(0,0)[rb]{\(\langle 9\rangle\)}}

\put(-4.1,-5.9){\makebox(0,0)[rb]{\(\langle 25\rangle\)}}
\put(-2.1,-5.9){\makebox(0,0)[rb]{\(\langle 27\rangle\)}}
\put(-0.1,-5.9){\makebox(0,0)[rb]{\(\langle 26\rangle\)}}

\put(0.5,-5){\makebox(0,0)[cc]{\(\mathcal{B}(4)\)}}
\put(2.1,-5.9){\makebox(0,0)[lb]{\(\langle 28\rangle\)}}
\put(4.1,-5.9){\makebox(0,0)[lb]{\(\langle 30\rangle\)}}
\put(6.1,-5.9){\makebox(0,0)[lb]{\(\langle 29\rangle\)}}

\put(-6.1,-7.9){\makebox(0,0)[rb]{\(\langle 98\rangle\)}}
\put(-4.1,-7.9){\makebox(0,0)[rb]{\(\langle 97\rangle\)}}
\put(-2.1,-7.9){\makebox(0,0)[rb]{\(\langle 96\rangle\)}}
\put(-0.1,-7.9){\makebox(0,0)[rb]{\(\langle 95\rangle\)}}

\put(1.0,-17.3){\framebox(6.2,11.7){}}
\put(0.9,-17.4){\framebox(6.4,11.9){}}

\put(0.5,-7){\makebox(0,0)[cc]{\(\mathcal{B}(5)\)}}
\put(2,-7.9){\makebox(0,0)[lb]{\(\langle 100\rangle\)}}
\put(4,-7.9){\makebox(0,0)[lb]{\(\langle 99\rangle\)}}
\put(6,-7.9){\makebox(0,0)[lb]{\(\langle 101\rangle\)}}

\put(-4.1,-9.9){\makebox(0,0)[rb]{\(\langle 388\rangle\)}}
\put(-2.1,-9.9){\makebox(0,0)[rb]{\(\langle 387\rangle\)}}
\put(-0.1,-9.9){\makebox(0,0)[rb]{\(\langle 386\rangle\)}}

\put(0.5,-9){\makebox(0,0)[cc]{\(\mathcal{B}(6)\)}}
\put(2.1,-9.9){\makebox(0,0)[lb]{\(\langle 390\rangle\)}}
\put(4.1,-9.9){\makebox(0,0)[lb]{\(\langle 389\rangle\)}}
\put(6.1,-9.9){\makebox(0,0)[lb]{\(\langle 391\rangle\)}}

\put(-5.8,-11.9){\makebox(0,0)[rb]{\(\langle 2224\rangle\)}}
\put(-4,-11.9){\makebox(0,0)[rb]{\(\langle 2223\rangle\)}}
\put(-1.8,-11.9){\makebox(0,0)[rb]{\(\langle 2222\rangle\)}}
\put(-0.1,-11.9){\makebox(0,0)[rb]{\(\langle 2221\rangle\)}}

\put(0.5,-11){\makebox(0,0)[cc]{\(\mathcal{B}(7)\)}}
\put(1.8,-11.9){\makebox(0,0)[lb]{\(\langle 2226\rangle\)}}
\put(3.8,-11.9){\makebox(0,0)[lb]{\(\langle 2225\rangle\)}}
\put(5.8,-11.9){\makebox(0,0)[lb]{\(\langle 2227\rangle\)}}

\put(0.5,-13){\makebox(0,0)[cc]{\(\mathcal{B}(8)\)}}
\put(0.5,-15){\makebox(0,0)[cc]{\(\mathcal{B}(9)\)}}

\put(0.1,-16.2){\makebox(0,0)[ct]{\(G^e_0(0,0)\)}}
\put(-2,-16.2){\makebox(0,0)[ct]{\(G^e_0(0,1)\)}}
\put(-4,-16.2){\makebox(0,0)[ct]{\(G^e_0(1,0)\)}}
\put(-6,-16.2){\makebox(0,0)[ct]{\(G^e_0(-1,0)\)}}
\put(2,-16.2){\makebox(0,0)[ct]{\(G^e_1(0,-1)\)}}
\put(4,-16.2){\makebox(0,0)[ct]{\(G^e_1(0,0)\)}}
\put(6,-16.2){\makebox(0,0)[ct]{\(G^e_1(0,1)\)}}

\put(-3.5,-17.7){\makebox(0,0)[ct]{with abelian maximal subgroup}}
\put(4.3,-17.7){\makebox(0,0)[ct]{without abelian maximal subgroup}}

\put(2.5,0){\makebox(0,0)[cc]{\textbf{TKT}}}
\put(3.5,0){\makebox(0,0)[cc]{a.1}}
\put(2.5,-0.5){\makebox(0,0)[cc]{\(\varkappa_s=\)}}
\put(3.5,-0.5){\makebox(0,0)[cc]{\((0000)\)}}
\put(1.8,-0.7){\framebox(2.6,1){}}
\put(-6,-2){\makebox(0,0)[cc]{\textbf{TKT}}}
\put(-5,-2){\makebox(0,0)[cc]{A.1}}
\put(-6,-2.5){\makebox(0,0)[cc]{\(\varkappa_s=\)}}
\put(-5,-2.5){\makebox(0,0)[cc]{\((1111)\)}}
\put(-6.7,-2.7){\framebox(2.6,1){}}

\put(-8,-19){\makebox(0,0)[cc]{\textbf{TKT}}}
\put(0,-19){\makebox(0,0)[cc]{a.1\({}^\ast\)}}
\put(-2,-19){\makebox(0,0)[cc]{a.2}}
\put(-4,-19){\makebox(0,0)[cc]{a.3}}
\put(-6,-19){\makebox(0,0)[cc]{a.3}}
\put(2,-19){\makebox(0,0)[cc]{a.1}}
\put(4,-19){\makebox(0,0)[cc]{a.1}}
\put(6,-19){\makebox(0,0)[cc]{a.1}}
\put(-8,-19.5){\makebox(0,0)[cc]{\(\varkappa_s=\)}}
\put(0,-19.5){\makebox(0,0)[cc]{\((0000)\)}}
\put(-2,-19.5){\makebox(0,0)[cc]{\((1000)\)}}
\put(-4,-19.5){\makebox(0,0)[cc]{\((2000)\)}}
\put(-6,-19.5){\makebox(0,0)[cc]{\((2000)\)}}
\put(2,-19.5){\makebox(0,0)[cc]{\((0000)\)}}
\put(4,-19.5){\makebox(0,0)[cc]{\((0000)\)}}
\put(6,-19.5){\makebox(0,0)[cc]{\((0000)\)}}
\put(-8.7,-19.7){\framebox(15.4,1){}}

\put(-4,-4){\oval(1.5,2)}
\multiput(-6,-4)(4,0){2}{\oval(1.5,1.5)}
\multiput(-5,-8)(0,-4){3}{\oval(3.6,1.5)}
\multiput(-2,-8)(0,-4){3}{\oval(1.5,1.5)}

\multiput(2,-8)(0,-4){3}{\oval(1.5,1.5)}
\multiput(4,-8)(0,-4){3}{\oval(1.5,1.5)}
\multiput(6,-8)(0,-4){3}{\oval(1.5,1.5)}

\put(-6,-5.1){\makebox(0,0)[cc]{\underbar{\textbf{32\,009}}}}
\put(-4,-5.2){\makebox(0,0)[cc]{\underbar{\textbf{142\,097}}}}
\put(-2,-5.1){\makebox(0,0)[cc]{\underbar{\textbf{72\,329}}}}
\put(-6,-9.1){\makebox(0,0)[cc]{\underbar{\textbf{494\,236}}}}
\put(-2,-9.1){\makebox(0,0)[cc]{\underbar{\textbf{790\,085}}}}
\put(2,-9.1){\makebox(0,0)[cc]{\underbar{\textbf{152\,949}}}}
\put(4,-9.1){\makebox(0,0)[cc]{\underbar{\textbf{62\,501}}}}
\put(6,-9.1){\makebox(0,0)[cc]{\underbar{\textbf{252\,977}}}}
\put(-6,-13.1){\makebox(0,0)[cc]{\underbar{\textbf{10\,200\,108}}}}
\put(-2,-13.1){\makebox(0,0)[cc]{\underbar{\textbf{14\,458\,876}}}}
\put(2,-13.1){\makebox(0,0)[cc]{\underbar{\textbf{27\,780\,297}}}}
\put(4,-13.1){\makebox(0,0)[cc]{\underbar{\textbf{10\,399\,596}}}}
\put(6,-13.1){\makebox(0,0)[cc]{\underbar{\textbf{2\,905\,160}}}}
\put(-6,-17.1){\makebox(0,0)[cc]{\underbar{\textbf{208\,540\,653}}}}
\put(-2,-17.1){\makebox(0,0)[cc]{\underbar{\textbf{37\,304\,664}}}}
\put(2,-17.1){\makebox(0,0)[cc]{\underbar{\textbf{62\,565\,429}}}}
\put(4,-17.1){\makebox(0,0)[cc]{\underbar{\textbf{63\,407\,037}}}}
\put(6,-17.1){\makebox(0,0)[cc]{\underbar{\textbf{40\,980\,808}}}}

\end{picture}

}

\end{figure}
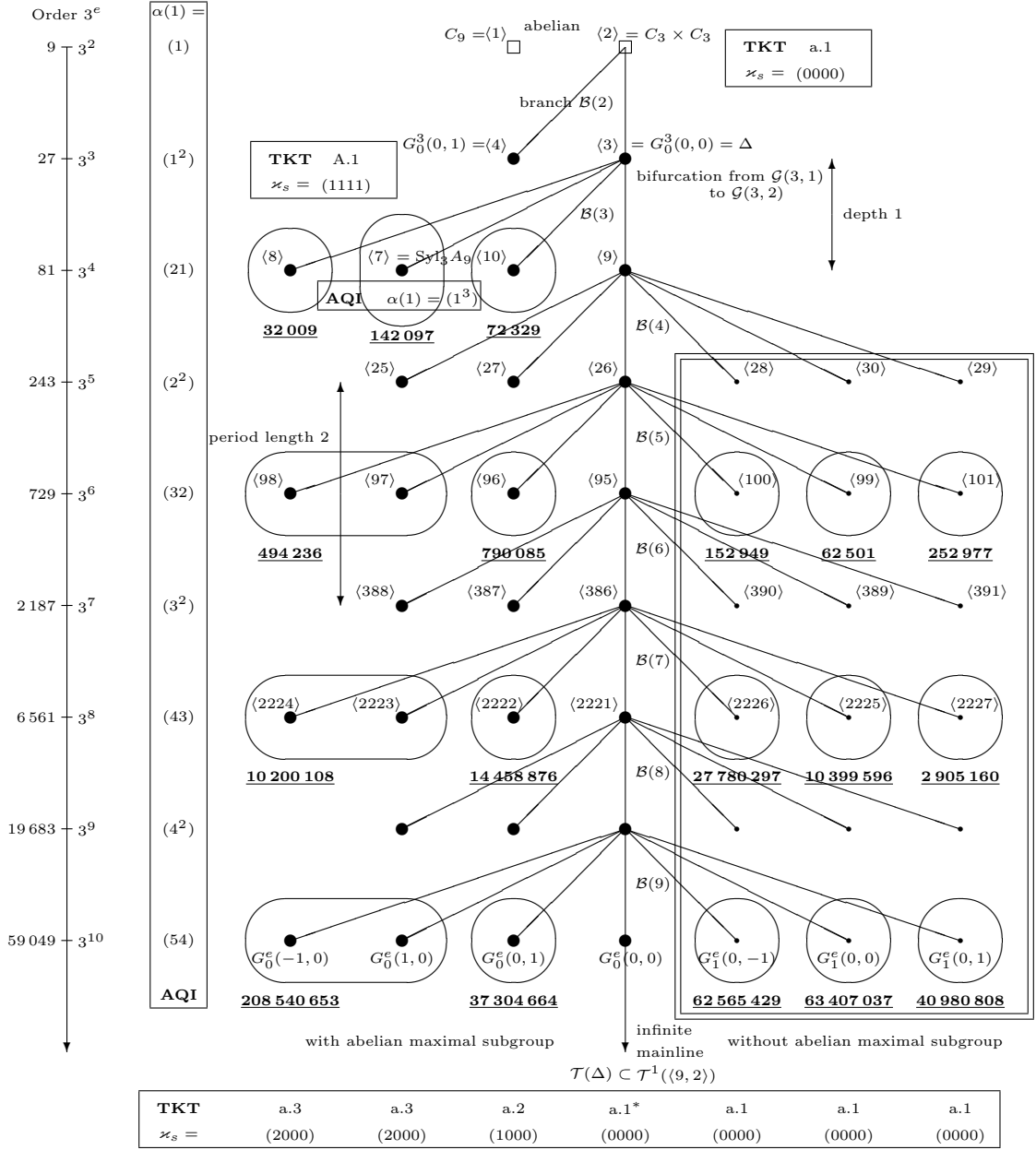

\newpage

\subsection{Entropy of unbounded coclass}
\label{ss:UnboundedCoClass}
Now we compare the two different measures
\({\bf P}_i\) (imaginary) and \({\bf P}_r\) (real)
of a descendant tree
\(\mathcal{T}(N)\)
with periodic bifurcations and consequently with unbounded coclass,
which is realized by both, imaginary-quadratic and real-quadratic fields.
For this purpose we need the \textit{imaginary probability-measure} of Boston, Bush, Hajir (2017)
\cite{BBH2017}:
\begin{equation}
\label{eqn:ImagMeasure}
{\bf P}_i(G)=\frac{y(G)^{g}}{\#\mathrm{Aut}(G)}\cdot(p^g)^{g}\cdot\prod_{k=1}^{g}\,\left(1-\frac{1}{p^k}\right)^2.
\end{equation}
In the special case
of the smallest odd prime number \(p=3\)
and the generator-rank \(g=2\),
which is relevant for our goal,
the constant factor after the invariants depending on \(G\) is here given by:
\[\frac{2^8}{3^2}.\]
The \textit{conversion-factor} between the two measures for \textbf{Schur \(\sigma\)-groups} (!) is in general:
\begin{equation}
\label{eqn:Factor}
{\bf P}_r(G)={\bf P}_i(G)\cdot\frac{y(G)}{\#G}\cdot\frac{p^{g+1}-1}{p-1}.
\end{equation}
In particular, for \(p=3\), \(g=2\), the constant factor (independent of \(G\)) is simply the number \(13\).
\begin{theorem}
\label{thm:UnboundedCoClass}
The \textbf{entropy} of the descendant tree
\(\mathcal{T}(N)\)
of all non-metabelian \(3\)-groups \(G\) with Artin-pattern
\((\alpha(G),\varkappa(G))
=((\lbrack 3,3\rbrack;\lbrack 3,3,3\rbrack^3,\lbrack 9,3\rbrack),(4111))\)
and \textbf{unbounded coclass},
\(\mathrm{cc}(G)\ge 2\),
which are descendants of Ascione's non-CF-group
\(N=\langle 729,45\rangle\)
\cite{BaBu2007},
\cite{Ma2017},
with respect to the normalized relative-measure
\({\bf P}_{r,rel}(G)=\frac{3^{10}}{2^6\cdot 13}\cdot{\bf P}_{r}(G)\)
associated to the real probability-measure \({\bf P}_{r}(G)\) in Formula
\eqref{eqn:RealMeasure}
is given by
\begin{equation}
\label{eqn:EntropyReal}
H({\bf P}_{r,rel})
=-\sum\,{\bf P}_{r,rel}(G_n^i)\log({\bf P}_{r,rel}(G_n^i)
=\frac{1}{26}\biggl(63\log(3)-42\log(2)\biggr)
\approx{\bf 1.54232},
\end{equation}
whereas, according to
\cite{MMSM2026},
with respect to the normalized relative-measure
\({\bf P}_{i,rel}(G)=\frac{3^6}{2^6}\cdot{\bf P}_{i}(G)\)
associated to the imaginary probability-measure \({\bf P}_{i}(G)\) in Formula
\eqref{eqn:ImagMeasure}
it is given by
\begin{equation}
\label{eqn:EntropyImag}
H({\bf P}_{i,rel})
=-\sum\,{\bf P}_{i,rel}(G_n^i)\log({\bf P}_{i,rel}(G_n^i))
=\frac{3}{2}\log(3)-\log(2)
\approx{\bf 0.9548}.
\end{equation}
\end{theorem}
\begin{proof}
In order to analyze the descendant tree \(\mathcal{T}(N)\), we need the \textit{real probability-measure}
in Formula
\eqref{eqn:RealMeasure}
and the Schur and (Schur\(+1\)) \(\sigma\)-descendants \(G\) of \(N\) in the Table
\ref{tbl:FormationLawsUnbounded}
\cite{BEO2005},
\cite{MAGMA6561}. \\
Using the Formula
\eqref{eqn:RealMeasure},
for \(r=g+1\),
and the information in Table
\ref{tbl:FormationLawsUnbounded},
we calculate the measures:
\({\bf P}_r(G)=\frac{y(G)^{2+1}}{\#\mathrm{Aut}(G)\cdot\#G}\cdot\frac{2^9\cdot 13}{3^3}\), grouped by families: 
270: \(\frac{(3^{n+2})^3}{2^2\cdot 3^{3n+8}\cdot 3^{3n+7}}\cdot\frac{2^9\cdot 13}{3^3}=\frac{2^7\cdot 13}{3^{3n+12}}\),
271: \(\frac{(3^{n+2})^3}{2^2\cdot 3^{3n+9}\cdot 3^{3n+7}}\cdot\frac{2^9\cdot 13}{3^3}=\frac{2^7\cdot 13}{3^{3n+13}}\),
272: \(\frac{(3^{n+2})^3}{2\cdot 3^{3n+9}\cdot 3^{3n+7}}\cdot\frac{2^9\cdot 13}{3^3}=\frac{2^8\cdot 13}{3^{3n+13}}\),
273: \(\frac{(3^{n+2})^3}{2\cdot 3^{3n+8}\cdot 3^{3n+7}}\cdot\frac{2^9\cdot 13}{3^3}=\frac{2^8\cdot 13}{3^{3n+12}}\).
For \(r=g\), however, we have 606:
\(\frac{(3^{n+2})^3}{2\cdot 3^{3n+9}\cdot 3^{3n+8}}\cdot\frac{2^8\cdot 13}{3^2}=\frac{2^7\cdot 13}{3^{3n+13}}\),
which coincides with 271.
\renewcommand{\arraystretch}{1.2}
\begin{table}[ht]
\caption{Parametrized formation-laws of invariants, grouped by families}
\label{tbl:FormationLawsUnbounded}
\begin{center}
\begin{tabular}{|c||c||c|c|c|c|}
\hline
 Family                    & \(606\)                           &  \(270\)                          &  \(271\)                          &  \(272\)                          &  \(273\)                          \\
\hline
 \(r\)                     & \(2\) (Schur)                     & \(3\) (Schur\(+1\))               & \(3\) (Schur\(+1\))               & \(3\) (Schur\(+1\))               & \(3\) (Schur\(+1\))               \\
 State                     & \(n\ge 0\), \(1\) group           & \(n\ge 0\), \(1\) group           & \(n\ge 0\), \(1\) group           & \(n\ge 0\), \(1\) group           & \(n\ge 0\), \(1\) group           \\
 \(y(G_n^i)\)              & \(3^{n+2}\)                       & \(3^{n+2}\)                       & \(3^{n+2}\)                       & \(3^{n+2}\)                       & \(3^{n+2}\)                       \\
 \(\#\mathrm{Aut}(G_n^i)\) & \(2\cdot 3^{3n+9}\)               & \(2^2\cdot 3^{3n+8}\)             & \(2^2\cdot 3^{3n+9}\)             & \(2\cdot 3^{3n+9}\)               & \(2\cdot 3^{3n+8}\)               \\
 \(\#G_n^i\)               & \(3^{3n+8}\)                      & \(3^{3n+7}\)                      & \(3^{3n+7}\)                      & \(3^{3n+7}\)                      & \(3^{3n+7}\)                      \\
 \({\bf P}_r(G_n^i)\)      & \(\frac{2^7\cdot 13}{3^{3n+13}}\) & \(\frac{2^7\cdot 13}{3^{3n+12}}\) & \(\frac{2^7\cdot 13}{3^{3n+13}}\) & \(\frac{2^8\cdot 13}{3^{3n+13}}\) & \(\frac{2^8\cdot 13}{3^{3n+12}}\) \\
\hline
\end{tabular}
\end{center}
\end{table}

\noindent
Now we add all probability-measures
of the descendant tree \(\mathcal{T}(N)\),
grouped by families,
and using the geometric series
\(\sum_{n=0}^{\infty}\,\frac{1}{27^n}
=\frac{1}{1-\frac{1}{27}}
=\frac{1}{\frac{26}{27}}
=\frac{27}{26}
=\frac{3^3}{2\cdot 13}\).
\begin{equation}
\label{eqn:270}
1 \text{ times 270}: \quad
\sum_{n=0}^{\infty}\,{\bf P}_r(G_n^0)=
\sum_{n=0}^{\infty}\,\frac{2^7\cdot 13}{3^{3n+12}}=
\frac{2^7\cdot 13}{3^{12}}\cdot\sum_{n=0}^{\infty}\,\left(\frac{1}{3^3}\right)^n=
\frac{2^7\cdot 13}{3^{12}}\cdot\frac{3^3}{2\cdot 13}=
\frac{2^6}{3^9},
\end{equation}
\begin{equation}
\label{eqn:271}
\text{606 and 271}: \quad
\sum_{n=0}^{\infty}\,{\bf P}_r(G_n^1)=
\sum_{n=0}^{\infty}\,\frac{2^7\cdot 13}{3^{3n+13}}=
\frac{2^7\cdot 13}{3^{13}}\cdot\sum_{n=0}^{\infty}\,\left(\frac{1}{3^3}\right)^n=
\frac{2^7\cdot 13}{3^{13}}\cdot\frac{3^3}{2\cdot 13}=
\frac{2^6}{3^{10}},
\end{equation}
\begin{equation}
\label{eqn:272}
1 \text{ times 272}: \quad
\sum_{n=0}^{\infty}\,{\bf P}_r(G_n^2)=
\sum_{n=0}^{\infty}\,\frac{2^8\cdot 13}{3^{3n+13}}=
\frac{2^8\cdot 13}{3^{13}}\cdot\sum_{n=0}^{\infty}\,\left(\frac{1}{3^3}\right)^n=
\frac{2^8\cdot 13}{3^{13}}\cdot\frac{3^3}{2\cdot 13}=
\frac{2^7}{3^{10}},
\end{equation}
\begin{equation}
\label{eqn:273}
1 \text{ times 273}: \quad
\sum_{n=0}^{\infty}\,{\bf P}_r(G_n^3)=
\sum_{n=0}^{\infty}\,\frac{2^8\cdot 13}{3^{3n+12}}=
\frac{2^8\cdot 13}{3^{12}}\cdot\sum_{n=0}^{\infty}\,\left(\frac{1}{3^3}\right)^n=
\frac{2^8\cdot 13}{3^{12}}\cdot\frac{3^3}{2\cdot 13}=
\frac{2^7}{3^9}.
\end{equation}
Finally we add these four contributions to the
\textbf{total measure of the tree}, taking into account their multiplicities:
\begin{equation}
\label{eqn:cc2}
\sum\,{\bf P}_{r}(G_n^i)=
\frac{2^6}{3^9}+2\cdot\frac{2^6}{3^{10}}+\frac{2^7}{3^{10}}+\frac{2^7}{3^9}=
\frac{2^6}{3^{10}}\cdot(3+2+2+6)=
\frac{2^6\cdot 13}{3^{10}}=
\frac{\bf 832}{\bf 59049}.
\end{equation}
With this result, we must normalize the probability-measure onto sum one,
\({\bf P}_{rel}(G)=\frac{3^{10}}{2^6\cdot 13}\cdot{\bf P}_{r}(G)\):
\renewcommand{\arraystretch}{1.2}
\begin{table}[ht]
\caption{Normalized relative-measure and logarithms, grouped by families}
\label{tbl:RelativeMeasuresUnbounded}
\begin{center}
\begin{tabular}{|c||c|c|c|c|}
\hline
 Family                          & \(270\)                  &  \(606,271\)             &  \(272\)                 &  \(273\)                 \\
\hline
 State                           & \(n\ge 0\), \(1\) group  & \(n\ge 0\), \(2\) groups & \(n\ge 0\), \(1\) group  & \(n\ge 0\), \(1\) group  \\
 \({\bf P}_{rel}(G_n^i)\)        & \(\frac{2}{3^{3n+2}}\)   & \(\frac{2}{3^{3n+3}}\)   & \(\frac{2^2}{3^{3n+3}}\) & \(\frac{2^2}{3^{3n+2}}\) \\
 \(-\log({\bf P}_{rel}(G_n^i))\) & \((3n+2)\log3-\log2\)    & \((3n+3)\log3-\log2\)    & \((3n+3)\log3-2\log2\)   & \((3n+2)\log3-2\log2\)   \\
\hline
\end{tabular}
\end{center}
\end{table}

\noindent
We can now calculate the \textbf{entropy of the unbounded coclass}:
{\small
\begin{equation*}
\label{eqn:Entropy}
\begin{aligned}
H({\bf P}_{rel})
&=-\sum\,{\bf P}_{rel}(G_n^i)\log({\bf P}_{rel}(G_n^i)) \\
&=\sum_{n=0}^{\infty}\,\frac{2}{3^{3n+2}}\biggl((3n+2)\log3-\log2\biggr)
+2\sum_{n=0}^{\infty}\,\frac{2}{3^{3n+3}}\biggl((3n+3)\log3-\log2\biggr) \\
&+\sum_{n=0}^{\infty}\,\frac{2^2}{3^{3n+3}}\biggl((3n+3)\log3-2\log2\biggr)
+\sum_{n=0}^{\infty}\,\frac{2^2}{3^{3n+2}}\biggl((3n+2)\log3-2\log2\biggr) \\
&=\frac{2}{9}\biggl(3\log3\cdot\sum_{n=0}^{\infty}\,\frac{n}{27^n}+(2\log3-\log2)\cdot\sum_{n=0}^{\infty}\,\frac{1}{27^n}\biggr) \\
&+\frac{4}{27}\biggl((3\log3+3\log3)\cdot\sum_{n=0}^{\infty}\,\frac{n}{27^n}+(3\log3-\log2+3\log3-2\log2)\cdot\sum_{n=0}^{\infty}\,\frac{1}{27^n}\biggr) \\
&+\frac{4}{9}\biggl(3\log3\cdot\sum_{n=0}^{\infty}\,\frac{n}{27^n}+(2\log3-2\log2)\cdot\sum_{n=0}^{\infty}\,\frac{1}{27^n}\biggr) \\
&=\biggl(\frac{2}{3}+\frac{8}{9}+\frac{4}{3}\biggr)\log3\cdot\frac{27}{26^2}
+\biggl(\frac{2}{9}(2\log3-\log2)+\frac{12}{27}(2\log3-\log2)+\frac{8}{9}(\log3-\log2)\biggr)\cdot\frac{27}{26} \\
&=\frac{1}{26}\biggl(63\log3-42\log2\biggr)
\approx{\bf 1.54232}. \qedhere 
\end{aligned}
\end{equation*}
}
\end{proof}


\section{Acknowledgements}
\label{s:Acknowledgements}

\noindent
We thank our academic instructor,
the Austrian mathematician Daniel C. Mayer,
for the suggestion
to apply the concept of entropy
to infinite probability distributions,
arising from \(3\)-class field tower groups
of real quadratic number fields.
 


\end{document}